\pdfoutput=1
\documentclass{amsart}

\usepackage{setspace}
\usepackage{amsmath}
\usepackage{graphicx}
\usepackage{amssymb}
\usepackage{amsthm}
\usepackage{verbatim}
\usepackage{enumerate}

\newtheorem{theorem}{Theorem}
\newtheorem{lemma}[theorem]{Lemma}

\newtheorem{prop}[theorem]{Proposition}
\newtheorem{corr}[theorem]{Corollary}

\theoremstyle{definition}
\newtheorem{mydef}[theorem]{Definition}
\newtheorem{remark}[theorem]{Remark}

\numberwithin{equation}{section}
\begin{document}
\title{ A Polynomial Variation of Meinardus' Theorem}

\author{Daniel Parry}
\address{Department of Mathematcs \\ Drexel University \\ Philadelphia, PA 19104}
\email{dtp29@drexel.edu}


\date{\today}

\begin{abstract}We develop a polynomial analogue of Meinardus' Thoerem for bivariate Euler products and apply it to the study of complex multiplicatively weighted partitions. \end{abstract}
\maketitle

A partition $\lambda=(\lambda_1,\lambda_2,...\lambda_k)$ of $n$ is a weakly decreasing sequence of positive integers whose sum is $n.$ Let $\{\mu_i\}$ be a sequence of complex numbers and let $p_w(n)$ count the number of partitions of $n$ under the complex multiplicative weight $w.$ That is, $\lambda$ is counted with weight $w(\lambda_1,\lambda_2,...\lambda_k) = \prod_{i=1}^k \mu_{\lambda_i}.$ Multiplicative weights appear to have been introduced in \cite{Wright01011934} and are a generalization of the classical partition counting problem. Letting $\mu_i= 1_{S}$ for $S\subset \mathbb{N}$ be the indicator function with support on $S$ forces $p_w(n)$ to equal the number of partitions of $n$ from a set $S.$

In general $p_w(n)$ is generated by the function \[f(q)=1+  \sum_{n=1}^\infty p_w(n)q^n = \prod_{m=1}^\infty \frac{1}{1-\mu_m q^m}. \]   If we let $\mu_i=z1_{S}$ we obtain the formula \begin{equation}\label{generatingfunction1}f(q)=1+  \sum_{n=1}^\infty p_w(n)q^n = \prod_{m\in S} \frac{1}{1-z q^m}. \end{equation}  In this case, $p_w(n)$ becomes a polynomial of degree $\le n$ in $z$ and $w(\lambda_1,\lambda_2,...\lambda_k)=z^k.$  For instance, if $\{1,2,3,4\}\subset S$ then $p_w(4) =z^4 + z^3 + 2z^2+z.$ Polynomials with generating functions similar to Equation \ref{generatingfunction1} have been studied by several authors \cite{2004math.....11377M, Wright01011934, PhysRevLett.79.1789, MR1855788, MR2880680, MR2427666, MR1631751}.  Motivated by applications in probability, some of the above writers have developed approximations to some of these polynomials when $z>0.$  However, developing approximations to $p_w(n)$ for complex $z$ has only been done in a few cases \cite{MR2880680, MR2427666, ParryBoyerPhase}.

When $z=1,$ there are a variety of formulas and techniques that have been developed to estimate $p_w(n).$  One popular technique we focus on is the theorem of Meinardus \cite{MeinardusTheorem} (See Theorem 6.2 in \cite{MR1634067} for a translated version). Meinardus essentially relates the asymptotic formulae of the coefficients \[1+\sum_{n=1}^\infty r(n)q^n = \prod_{m=1}^\infty \frac{1}{(1-q^m)^{a_m} }\] to the analytic behavior of the Dirichlet series $D(s)= \sum_{n=1}^\infty  a_n / n^s$ and the Fourier series $g(\tau) = \sum_{k=0}^\infty  a_k q^k. $ For this paper, assume the standard notation $s=\sigma +it,$ $\tau= 2\pi \alpha -i 2\pi \psi,$ and $q=e^{-\tau}.$ 

To be precise, suppose:
\begin{enumerate}  
\item Assume $D(s)$ converges for $\sigma >s_0.$
\item For some $\sigma_0\in (-1,0),$ $D(s)$  has a meromorphic continuation to $\sigma\geq \sigma_0$ with a simple pole at $s_0$ with residue $A.$ 
\item There is a $C>0$ so that as $|t|\to \infty$  $D(s)=O(|t|^{C}).$
\item For $|\arg \tau|>\pi/4,$ there is a $\epsilon>0$ and $C'>0$ so that as $|\alpha |\to 0$ \[\Re g(\tau)-g(2\pi \alpha)\le -C' |\alpha|^{-\epsilon}.\]

\end{enumerate}
then \begin{theorem}  As $n\to \infty$ \[r(n)=Cn^\kappa \exp\left(\frac{s_0+1}{s_0} n^\frac{s_0}{s_0+1} [A\Gamma(s_0+1)\zeta(s_0 +1)]^\frac{1}{s_0+1} )\right)(1+O(n^{-\kappa_1}))\] where $\zeta(s)= \sum_{m=1}^\infty m^{-s}$ is the Riemann zeta function, and 
\begin{align*}
C&= e^{D'(0)}[2\pi(1+s_0)]^{-\frac{1}{2}}[A\Gamma(s_0+1)\zeta(s_0+1)]^\frac{1-2D(0)}{2s_0+2} \\
\kappa &=\frac{D(0)-1-\frac{1}{2}s_0 }{1+s_0} \\
\kappa_1 &=\frac{s_0}{s_0+1}\min \left(\frac{-\sigma_0}{s_0}-\frac{\delta}{4}, \frac{1}{2}-\delta \right)
\end{align*} $\delta$ an arbitrary real number.  
\end{theorem}  By letting $a_m=1_S$ we are able to compute $p_w(n)$ for a class of $w.$ Recently, a few authors have made variations and improvements \cite{MR2449593, 2007math......1584G} on Meinardus' result.  This paper will continue in that direction.  This paper will apply the circle method to develop an analogue of Meinardus' theorem for polynomials of the form \[P(z,q)=1+\sum_{n=1}^\infty Q_n(z)q^n = \prod_{m=1}^\infty \frac{1}{(1-zq^m)^{a_m} }\] when $z\in \mathbb{D},$  where $\mathbb{D}$ is the open unit disk and $a_m\in \mathbb{R}.$ Hence by the same idea of  letting $a_m=1_S$  we will be able to estimate $p_w(n)$ for a large class of weights $w.$ Concurrently, this paper provides the key step in generalizing the results in  \cite{MR2880680, MR2427666}.

\section{Statement of Main Theorem}
Consider \[P(z,q)=1+\sum_{n=1}^\infty Q_n(z)q^n = \prod_{m=1}^\infty \frac{1}{(1-zq^m)^{a_m} }\] for $z\in \mathbb{D}.$ Define the family of twisted Dirichlet series  $\{D_{h,k}(s) \}$ with $h,k\in \mathbb{N}$ \[D_{h,k}(s) = \sum_{m=1}^\infty \frac{e^{2\pi i \frac{hm}{k}}a_m}{m^s}.\] Fix $c,s_0>0,$ $-1<\sigma_0<0,$ and assume \begin{enumerate} \item Each Dirichlet series $D_{h,k}(s)$ converges uniformly and absolutely in some half plane $\sigma >c>0. $   
\item Each $D_{h,k}(s)$ has a meromorphic continuation to $\sigma\geq \sigma_0$ with a simple pole at $0<s_0<c$ with residue $A_{h,k}.$ 
\item There exists constants $C_1(\sigma_0,c), C_2(\sigma_0,c)>0$ so that for each $\sigma\in [\sigma_0,c]$ \[|(s-s_0)D_{h,k}(\sigma+it)|\le C_1(1+|t|)^{C_2} k^{s_0+|\sigma|}.\]
    \newcounter{enumTemp}
    \setcounter{enumTemp}{\theenumi}

\end{enumerate}
Both $D_{h,k}(0)$ and $A_{h,k}$ are periodic functions of $h$ with $k$ fixed so we can define discrete Fourier expansions \[D_{h,k}(0) = \sum_{j\in \mathbb{Z}_k}e^{2\pi i \frac{hj}{k}}b(j), \qquad A_{h,k}= \sum_{j\in \mathbb{Z}_k}e^{2\pi i \frac{hj}{k}}c(j).\]
Likewise it is also useful to define \[J(s;z,r,h,k,w) = \Gamma(s) \Phi(z^k,s+1, \frac{r}{k}) D_{rh,k}(s)(k w)^{-s}\] where $\Phi(z,s,\nu)$ is the Lerch phi function\[\Phi(z,s,\nu)=\sum_{n=0}^\infty \frac{z^n}{(n+\nu)^s}\] and we define the polylogarithm $Li_s(z)=\Phi(z,s,1)$ as a special case.

\begin{prop} \label{Factorization} Let $z\in \mathbb{D}$ be the open unit disk, $h,k$ be relatively prime positive integers with $h\le k$ and $\Re w>0.$ Then 
\begin{align*}\ln(P(z,e^{2\pi i\frac{h}{k} -w} )) &= \Psi_{h,k}(z,w)+\ln \omega_{h,k,n}(z) + g_{h,k}(z,w)+ 2\pi in\frac{h}{k},
\end{align*}  where \begin{align*}
\Psi_{h,k}(z,w) &=\frac{\Gamma(s_0+1)}{s_0  w^{s_0}}\sum_{j\in \mathbb{Z}_k}c(j)Li_{s_0+1}(e^{2\pi i \frac{hj}{k}}z), \\
\omega_{h,k,n}(z)&=e^{-2\pi i \frac{hn}{k}} \prod_{j\in \mathbb{Z}_k}(1-e^{2\pi i \frac{hj}{k}}z)^{-b(j)},\\
g_{h,k}(z,w)&=\frac{1}{k}\sum_{r=1}^{k} z^r \frac{1}{2\pi i} \int_{\sigma_0-i\infty}^{\sigma_0+i\infty} J(s;z,r,h,k,w) ds. 
\end{align*}
\end{prop}

\begin{mydef}\label{DefPhase}Fix a domain $D\subset \mathbb{D}$ and for each pair of relatively prime positive integers $h,k$ with $0< h\le k$ we define  $L_{h,k}(z) = \sqrt[s_0+1]{w^{s_0}s_0\Psi_{h,k}(z,w)}.$  Here we define the $s_0+1$ root as on the principal branch with $\arg z \in (-\pi,\pi].$  We define $[L_{p,q}(z)]$ as the equivalence class of $\{L_{h,k}(z)\}$ whose real component is identical to $\Re L_{p,q}(z)$ on $D.$ The set $R(p,q)$ we call the $(p,q)$-th phase (or simply phase $(p,q)$) is defined by \[R(p,q)= \{z\in D:  L_{h,k}(z) \not \in [L_{p,q}(z)] \implies \Re L_{p,q}(z)> \Re L_{h,k}(z)\}.\]
\end{mydef} \noindent We require two additional hypothesize on these $L_{h,k}(z)$ functions to provide asymptotic approximations.  Let $X\subset R(p,q)$ be a compact set and   \begin{enumerate}
    \setcounter{enumi}{\theenumTemp}
\item There exists positive constants $L^m_X=\inf_{z\in X} \Re L_{p,q}(z),$ and $L^M_X = \sup_{z\in X} \Re L_{p,q}(z).$  
\item The function $\Re L_{h,k}(z)$ vanishes uniformly on $X$ as $k$ grows large.  
 \end{enumerate} 
\begin{remark}  The function $\Re L_{h,k}(z)$ can be interpreted as a measure of the relative strength of arc $q=e^{2\pi i \frac{h}{k}}$ in the circle method.  In this context, a phase is simply a set where the major arcs are well defined. \end{remark} 
\begin{remark} Unlike typical circle method calculations, it is not unusual for major arcs to be more than just $q=1.$  Even in the simplest of examples, say $a_m=1,$ we observe $z$ where major arcs could be $q=1, -1, e^{\pm 2\pi i \frac{1}{3}}$ or any combination of the three. This is the primary difficulty in this generalization.
 \end{remark} 

Under these conditions, we have an expansion of $Q_n(z)$ given by \[Q_n(z)=\sum_{\{(h,k) :\ L_{h,k}(z)\in [L_{p,q}(z)]\}} \omega_{h,k,n}(z) \tilde{I}_{h,k,n}(z)\] where $I_{h,k,n}(z)$ can be estimated by
  \begin{theorem}\label{FourierCoeff}
If $X\subset \{z\in R(p,q): w^{s_0}s_0\Psi_{p,q}(z, w )\not \le 0\}$  then  \begin{align*}\tilde{I}_{h,k,n}(z)&\thicksim_X  \sqrt{\frac{L_{h,k}(z)}{n^\frac{s_0+2}{s_0+1}2\pi (s_0+1)}} \exp\left( \frac{s_0+1}{s_0}n^\frac{s_0}{s_0+1} L_{h,k}(z)\right).
 \end{align*} If $X\subset \{z\in R(p,q): w^{s_0}s_0\Psi_{p,q}(z,w)\le 0\}$  then
 
 \begin{align*}\tilde{I}_{h,k,n}(z)&\thicksim_X 2\Re \left[\sqrt{\frac{L_{h,k}(z)}{n^\frac{s_0+2}{s_0+1}2\pi (s_0+1)}} \exp\left( \frac{s_0+1}{s_0}n^\frac{s_0}{s_0+1} L_{h,k}(z)\right)\right].
 \end{align*} \end{theorem} \begin{remark} To make the theorem statement concise we abuse the notation.  By saying $a_n\thicksim \Re b_n$ in Theorem \ref{FourierCoeff}, we actually mean $a_n=|b_n|(\cos(\arg b_n) +O(n^{-\mu})).$ Likewise, for both estimates the relative error is $O(n^{-\mu})$ where   \[\mu = \min \left( \frac{-\sigma_0}{s_0+1},\frac{s_0}{2s_0+2}\right). \]
 \end{remark}
 
 \begin{remark} The second approximation in Theorem \ref{FourierCoeff}, suggests that when $L_{h,k}(z)$ fails to be analytic, we expect $Q_n(z)$ should have a highly oscillatory behavior. 
 \end{remark}
  \section{Examples}  We say a sequence is admissible if it satisfies Assumptions (1-3).     The space of admissible sequences forms a vector space of infinite dimension graded by $s_0$ and $\sigma_0.$  Furthermore any sequence of finite support rests inside this space.  Multiplying any sequence $m^{s}$ shifts the $s_0$ grading. While it is not known which sequences are admissible and which are not, we can demonstrate this space is large.
\begin{lemma}\label{arethmeticzeta} For  every $c>\sigma>\sigma_0$ there exists constants, $C_1,C_2>0$ so that for every $1\le h\le k$  \[|\zeta(s,h/k)(s-1)|\le C_1(1+|t|)^{C_2}k^{\max(0,\sigma)} \] where $\zeta(s,\nu)=\sum_{n=0}^\infty (n+\nu)^{-s}$ is the Hurwitz zeta function.  
\end{lemma}
\begin{proof} It is an immediate consequence of Theorem 12.23 of \cite[Page 270]{MR0422157} that there exists constants $A,C_2>1$ dependent solely on $\sigma_0$ and $c$ so that \[\left|(s-1)\left(\zeta(s,h/k)-\left(\frac{h}{k}\right)^{-s}\right)\right|\le A(1+|t|)^{C_2}. \]  The theorem follows immediately from this fact.

\end{proof}
\begin{corr}  The following are true \begin{enumerate}  \item Any periodic sequence is admissible with $s_0=1$ and $\sigma_0 \in (-1,0).$
\item The sequence $a_m=1_{a\mod j}$ is admissible for every $a,j\in \mathbb{N}$ with $s_0=1$ and $\sigma_0\in (-1,0).$
\item Any eventually periodic sequence is admissible with $s_0=1$ and $\sigma_0 \in (-1,0).$
\item If $a_m$ is an eventually periodic sequence then $a_m m^{s-1}$ for $s>0$ is admissible with $s_0=s$ and $\sigma_0\in (-1,0).$   \end{enumerate} \end{corr}
\begin{proof}  Consider any periodic function$ \chi(m)$ on $\mathbb{Z}_j$ twisted with an additive character $\psi(m)=e^{\frac{2\pi i m h}{k}}$ on $\mathbb{Z}_k.$ Without loss of generality assume $|\chi(m)|\le M.$  The resulting function $\chi(m)\psi(m)$ is periodic with period $jk.$ We then can write \[D_{h,k}(s) = \sum_{r=1}^{jk}\frac{\chi(r)\psi(r)}{(jk)^s}\zeta(s,r/jk).  \]

It is well known $\zeta(s,\nu)$ continues analytically to the entire complex plane with at most a simple pole at $s=1$ with residue $1,$ and  $D_{h,k}(\sigma)\le M \pi^2/6$ for $\sigma>c>2,$ we need only demonstrate that the bound in Assumption (3) holds. This follows naturally from Lemma \ref{arethmeticzeta}

\[|D_{h,k}(\sigma+it)| \le M C_1(jk) (1+|t|)^{C_2} (jk)^{\max(-\sigma,0)}\le \hat{C_1} (1+|t|)^{C_2} k^{1+|\sigma|}. \]

Parts (2), (3)  now follow by linearity. Part (4) follows by observing that multiplying by $m^s$ shifts $\sigma \to \sigma +s.$

\end{proof} Admissible sequences are admissible in the sense that they produce the proper asymptotic approximation near roots of unity.  One requires assumption (4), (5) to hold as well in order to compute the Fourier coefficients.  We do conjecture that Assumptions (4) and (5) can be relaxed and/or possibly replaced with $a_m$ is not finitely supported.

    \subsection{The Constant Sequence}
      Consider the sequence $a_m=1$ identically.  In this case, $p_w(n)=Q_n(z)$ counts the total number of partitions of $n$ weighed so that each partition is counted with weight $z^{l(\lambda)}$ where $l(\lambda)$ counts the total number of summands for a partition of $n.$  These polynomials have been called the Partition Polynomials. Wright developed detailed asymptotics for $p_w(n)$ for  $z\in (0,1)$ \cite{Wright01011934} and the roots of these polynomials were studied in a work of Boyer and Goh \cite{BoyerPartitionPolynomials}.
      
Since $a_m$ is periodic with period 1, it is admissible with $s_0=1.$ After working through all the details, one will obtain  \[L_{h,k}(z)=  \frac{1}{k}\sqrt{ Li_2(z^k)}.\]   With $D= \mathbb{D}\setminus \{0\}$  we can define our phases, $R(h,k).$  While finding which $R(h,k)$ are nonempty is a nontrivial problem, it has been solved \cite{BoyerPartitionPolynomials}.  For this choice of $D$ there are 3 nonempty phases, $R(1,1),$  $R(1,2),$ and $R(1,3).$ The approximations given in \cite{BoyerPartitionPolynomials} follow naturally.

  \subsection{The Power Function Sequence}
A natural extension of the constant sequence is the power function sequence.  Consider the sequence $a_m=m^{s_0-1}$ with $s_0>0.$  When $s_0=1$ we obtain the Partition Polynomials (see above).  When $s_0=2,$ we obtain the Plane Partition Polynomials  \cite{MR2880680, ParryBoyerPhase}.

This sequence is admissible as it is a periodic sequence multiplied by a power function.  Constructing \[L_{h,k}(z)=  \frac{1}{k}\sqrt{ \Gamma(s_0+1)Li_{s_0+1}(z^k)}\]   with $D= \mathbb{D}\setminus \{0\}$  we can define our phases, $R(h,k).$ Assumptions (4) is an application of Diophantine approximation and assumption (5) is trivial.  
While one can prove that all but finitely many nonempty phases are empty for a given $s_0,$ computing the $R(h,k)$ analytically is a nontrivial problem and these sets are poorly understood.  What can be said comes from the work in \cite{ZeroAttractoram} which follows from generalizing the work in \cite{ParryBoyerPhase} and applying some numerical techniques. For example, one can prove the existence of constants $c_1\approx 2.148\dots$ and $c_2\approx 1.03\dots$ so that for $c_1>s_0>c_2$ only $R(1,1)$ and $R(1,2)$ are nonempty and for $s_0>c_1,$ $R(1,1)$ is the only nonempty phase.

If we let $D=(-1,1),$  the problem simplifies and one can show that there exists a $x^*(s_0)\in[-1,0]$ so that $R(1,1) =(0,1)\cup (0,x^*(s_0))$ and $R(1,2)=(-1,x^*(s_0)).$  Since $D_{1,1}(0)=\zeta(1-s_0)$ and $D_{1,2}(0) = (2^{s_0}-1)\zeta(1-s_0)$ we can compute $\omega_{1,1,n}(z)$ and $\omega_{1,2,n}(z)$ by Fourier inversion.
 \begin{theorem}  Suppose $X\subset (0,1)$ is compact and $x\in X$ then 
\[Q_n(x) \thicksim (1-x)^{-\zeta(1-s_0)}\sqrt{\frac{L_{1,1}(x)}{2\pi (s_0+1)n^\frac{s_0+2}{s_0+1}}}\exp\left(\frac{s_0+1}{s_0} n^\frac{s_0}{s_0+1} L_{1,1}(x)\right)\] where $L_{1,1}(x)=\sqrt[s_0+1]{\Gamma(s_0+1)Li_{s_0+1}(x)}.$\end{theorem}

 \begin{theorem}There exists an $x^*(s_0)\in [-1,0]$ so that if $X\subset (-1,x^*(s_0))$ is compact and $x\in X$ then 
\[Q_n(x) \thicksim(-1)^n(1-x)^{-B}(1+x)^{-A} \sqrt{\frac{L_{1,2}(x)}{2\pi (s_0+1)n^\frac{s_0+2}{s_0+1}}}\exp\left(\frac{s_0+1}{s_0} n^\frac{s_0}{s_0+1} L_{1,2}(x)\right)\]
where 
 \[A= \zeta(1-s_0)(1-2^{s_0-1}), \quad B =  \zeta(1-s_0)2^{s_0-1}, \quad L_{1,2}(x)=\frac{1}{2} \sqrt[s_0+1]{\Gamma(s_0+1)Li_{s_0+1}(x^2)}.\] \end{theorem}

\begin{theorem}\label{PowerfunctionNegativeAxis}  There exists an $x^*(s_0)\in [-1,0]$ so that if $X\subset (x^*(s_0),0)$ is compact and $x\in X$ then 
\[Q_n(x) \thicksim 2(1-x)^{-\zeta(1-s_0)}\Re \left[\sqrt{\frac{L_{1,1}(x)}{2\pi (s_0+1)n^\frac{s_0+2}{s_0+1}}}\exp\left(\frac{s_0+1}{s_0} n^\frac{s_0}{s_0+1} L_{1,1}(x)\right)\right]\] where $L_{1,1}(x)=\sqrt[s_0+1]{\Gamma(s_0+1)Li_{s_0+1}(x)}.$\end{theorem}

  \subsection{Arithmetic Progression Indicator Sequences}

Consider partitions whose parts all lie in an arithmetic progression, say $\lambda_i=a\mod j$ and $a,j$ are relatively prime positive integers with $j>1.$  Consider $Q_n(z)=p_w(n)$ to be the weighted count of all such partitions of $n$ where each partition is counted with weight $z^{l(\lambda)}.$ 

  These polynomials are generated by the sequence $a_m$ which is one only if $m=a \mod j$ and zero otherwise.  This sequence is admissible and \[L_{h,k}(z)= \frac{(k,j)}{k} \sqrt{\frac{1}{j} Li_{2}\left(z^{\frac{k}{(k,j)}}e^{\frac{2\pi i ha}{(k,j)}}\right)}\] for $(h,k)=1$ and undefined  otherwise.     With $D= \mathbb{D}\setminus \{0\}$  we can define our phases, $R(h,k).$ When $j>2,$ several reductions one can prove
\begin{theorem}  Suppose $X\subset \{z\in \mathbb{D}: |\arg z| < \frac{\pi}{j} \}$ is compact and $z\in X$ then 
\[Q_n(z) \thicksim (1-z)^\frac{2a-j}{2j} \sqrt[4]{\frac{Li_2(z)}{16j \pi^2n^3}}\exp\left( 2\sqrt{\frac{nLi_2(z)}{j}} \right).\]  \end{theorem} \noindent while $Q_n(e^{\frac{2\pi i a}{j}}z) =e^{\frac{2\pi i n}{j}}Q_n(z).$ When $j=2$ we obtain a slightly different reduction.     \begin{theorem}  Suppose $X\subset R(1,1)$ is compact and $z\in X$ then 
\[Q_{n}(z) \thicksim \sqrt[4]{\frac{Li_2(z)}{32 \pi^2n^3}}\exp\left( \sqrt{2nLi_2(z)} \right).\]
Suppose $X\subset R(1,2)$ is compact and $z\in X$ then 
\[Q_n(z) \thicksim (-1)^n \sqrt[4]{\frac{Li_2(-z)}{32 \pi^2n^3}}\exp\left( \sqrt{2nLi_2(-z)} \right).\]
Suppose $X\subset R(1,4)$ is compact and $z\in X$ then 
\[Q_n(z) \thicksim \left(i^{-n}\sqrt[4]{\frac{z-i}{z+i}}+i^{n}\sqrt[4]{\frac{z+i}{z-i}}\right) \sqrt[4]{\frac{Li_2(-z^2)}{128 \pi^2n^3}}\exp\left( \sqrt{\frac{nLi_2(-z^2)}{2}} \right).\]
\end{theorem}
\begin{remark}
The fact that $R(1,4)$ exists is unique to every other arithmetic progression.  It leads one to ask whether there is a combinatorial reason for $R(1,4)$'s existence.  
\end{remark}
 
 \section{Proof of Proposition \ref{Factorization}}
\noindent Start by expanding
\begin{equation}\label{expansion} \ln(P(z,e^{2\pi i\frac{h}{k} -w} )) = \sum_{l=1}^\infty \sum_{m=1}^\infty \frac{z^l}{l} a_m e^{2\pi i \frac{hlm}{k}} e^{-w lm}. \end{equation} Now apply the Cahen-Mellin integral for sufficiently large $c>0$  \[ e^{- w lm} = \frac{1}{2\pi i} \int_{c-i\infty}^{c+i\infty} \Gamma(s) (w lm)^{-s} ds. \]  We plug this into Equation \ref{expansion} and rearrange the terms.\begin{align*} \ln P(z,e^{2\pi i\frac{h}{k} -w}) &= \sum_{l=1}^\infty \sum_{m=1}^\infty \frac{z^l}{l} a_m e^{2\pi i \frac{hlm}{k}} \frac{1}{2\pi i} \int_{c-i\infty}^{c+i\infty} \Gamma(s) ( l m w)^{-s} ds. \\
&=\frac{1}{2\pi i} \int_{c-i\infty}^{c+i\infty} \Gamma(s) \sum_{l=1}^\infty \sum_{m=1}^\infty \frac{z^l a_m e^{2\pi i \frac{hlm}{k}} }{l^{s+1}m^{s}} w^{-s} ds. \end{align*}We now sum by letting $l=nk+r$ where $n\in \mathbb{N}\cup \{0\}$ and $1\le r\le k.$  Dividing top and bottom of our fraction by $k^{s+1}$ we can observe the Lerch phi function plays a natural role in approximating $P(z,q)$.
\begin{align*}\ln P(z,e^{2\pi i\frac{h}{k} -w}) &=\frac{1}{k}\sum_{r=1}^{k}z^r\frac{1}{2\pi i} \int_{c-i\infty}^{c+i\infty} \Gamma(s) \Phi(z^k,s+1, \frac{r}{k}) D_{rh,k}(s)( k w)^{-s} ds \\ 
&:=\frac{1}{k}\sum_{r=1}^{k}z^r\frac{1}{2\pi i} \int_{c-i\infty}^{c+i\infty} J(s;z,r,h,k, w)ds.
 \end{align*}
Shift the contour over to $\Re s =\sigma_0$ by Cauchy's Theorem.  There are two singularities for $J:$ a simple pole at $s=0$ created by $\Gamma(s)$ and a simple pole at $s=s_0$ created by $D_{h,k}(s)$ by assumption.  Calculating the residues accordingly, we obtain \begin{align*}\frac{1}{2\pi i} \int_{c-i\infty}^{c+i\infty} J(s;z,r,h,k,w)ds&= \frac{1}{2\pi i} \int_{\sigma_0-i\infty}^{\sigma_0+i\infty} J(s;z,r,h,k,w) ds + Res (J,s_0) + Res (J,0) \\
 Res(s_0;J) &= \Phi(z^k,s_0+1, \frac{r}{k})\frac{\Gamma(s_0) A_{rh,k}}{( k w)^{s_0}} \\
Res(0;J) &=\Phi(z^k,1, \frac{r}{k})D_{rh,k}(0). \end{align*} 

\noindent And thus we can define \begin{align*} \Psi_{h,k}(z,w) & =\frac{\Gamma(s_0+1)}{s_0 k^{s_0+1} w^{s_0}}\sum_{r=1}^{k} z^r\Phi(z^k,s_0+1, \frac{r}{k}) A_{rh,k} \\
\omega_{h,k,n}(z) & = \exp(\sum_{r=1}^{k} \frac{z^r}{k}\Phi(z^k,1, \frac{r}{k})D_{rh,k}(0) - 2n\pi i \frac{h}{k})\\
g_{h,k}(z,w)&=\frac{1}{k}\sum_{r=1}^{k} z^r \frac{1}{2\pi i} \int_{\sigma_0-i\infty}^{\sigma_0+i\infty} J(s;z,r,h,k,w)ds. \end{align*} To complete the proof, substitute $D_{h,k}(0)=\sum_{j\in \mathbb{Z}_k}e^{2\pi i \frac{hj}{k}}b(j)$  and use the identity $\sum_{r=1}^k z^r\Phi(z^k,s,r/k)=k^sLi_{s}(z).$  It follows that  \begin{align*}
\omega_{h,k,n}(z)
&=\exp(\sum_{j\in \mathbb{Z}_k}b(j)\sum_{r=1}^{k} \frac{(ze^{2\pi \frac{i h j}{k}})^r}{k}\Phi(z^k,1, \frac{r}{k}) - 2n\pi i \frac{h}{k})\\ 
&=\exp(\sum_{j\in \mathbb{Z}_k}-b(j)\ln(1-ze^{2\pi i \frac{hj}{k}}) - 2n\pi i \frac{h}{k}) \\
&=e^{-\frac{2\pi inh}{k}}\prod_{j\in \mathbb{Z}_k}(1-ze^{2\pi i \frac{hj}{k}})^{-b(j)}
\end{align*} and the formula for $\Psi_{h,k}(z,w )$ follows by the same method.  

What makes this approximation useful are the controls given in the following lemma.  
\begin{lemma}\label{Factorization:est}For $z\in X\subset\mathbb{D}$ compact define constants  $M_X=\sup_{z\in X}|z|,$ \[S=S(\sigma_0) =\sup_{t\in \mathbb{R}}\left(|\Gamma(\sigma_0+it)|e^{\pi|t|/2}(1+|t|)^{\frac{1}{2}-\sigma_0}\right), \quad C_3=\frac{S2 e^{\frac{\pi}{2}} C_1  \Gamma(C_2 -\sigma_0+\frac{1}{2})}{(1-M_X)^2} . \]  a) For every $\Re w >0,$ \[|g_{h,k}(z,w)| \le  C_3 k^{s_0}|kw |^{-\sigma_0}(\frac{|\Im w| }{\Re w}+1)^{C_2 -\sigma_0 +\frac{1}{2}} . \]
b) For $z\in X\subset \mathbb{D}$ compact \[\ln |\omega_{h,k,n}(z)| \le -C_1\ln(1-M_X) k^{s_0}. \] \end{lemma}
\begin{proof} a)  Because \begin{align*}|g_{h,k}(z,w)|\le \frac{1}{1-M_X}\sup_{1\le r\le k}\frac{1}{k}\int_{\sigma_0-i\infty}^{\sigma_0+i\infty} |J(s;z,r,h,k,w)| |ds|, \end{align*}    it suffices to prove that, \[ \int_{\sigma_0-i\infty}^{\sigma_0+i\infty} |J(s;z,r,h,k,w)| ds \le \frac{1}{1-M_X} \frac{2 e^{\frac{\pi}{2}} C_1 S  k^{s_0+1}|kw |^{-\sigma_0}}{(\frac{\pi}{2}-|\arg w|)^{C_2 -\sigma_0 +\frac{1}{2}}} \Gamma(c_2 -\sigma_0+\frac{1}{2}). \] First by Assumption (3), we know $|D_{h,k}(\sigma_0+it)|\le C_1(1+|t|)^{C_2}k^{s_0-\sigma_0}.$ Next, we estimate $\Phi(z,s,v)$ by breaking off the first term and bounding the remainder.  Hence, \[|\Phi(z^k,s+1,\frac{r}{k})|\le k^{\sigma_0+1} +\Phi(|z|^k,\sigma_0 + 1,1) \le k^{\sigma_0+1}(1+\Phi(M_X,\sigma_0 + 1,1))\le \frac{k^{\sigma_0+1}}{1-M_X}.\]  Therefore,  \begin{align*}\int_{-\infty}^\infty |J(\sigma_0+it;z,r,h,k,w)|dt &\le \frac{C_1 k^{s_0+1}}{1-M_X} \int_{-\infty}^\infty (1+|t|)^{C_2}|\Gamma(\sigma_0 +it)| |(kw)^{-\sigma_0-it}| dt
\end{align*} 
As a consequence of \cite[Corollary 1.4.4]{MR1688958}, there has to exist an $S>0$ so that \[|\Gamma(\sigma_0+it)e^{\pi|t|/2}(1+|t|)^{\frac{1}{2}-|\sigma_0|}| \le S.\] Therefore the integral \begin{align*}\int_{-\infty}^\infty (1+|t|)^{C_2}|\Gamma(\sigma_0 +it)| |(kw)^{-\sigma_0-it}| dt &\le S|kw|^{-\sigma_0} \int_{\infty}^\infty (1+|t|)^{C_2-\sigma_0 -\frac{1}{2}}e^{-|t|(\frac{\pi}{2} -|\arg w|)}dt \\
%&=\frac{S|kw|^{-\sigma} \int_{\infty}^\infty ( \frac{\pi}{2} -|\arg \alpha| +|u|)^{C_2-\sigma_0 -\frac{1}{2}}e^{-|u| }du}{(\frac{\pi}{2} - |\arg w|)^{C_2-\sigma_0 +\frac{1}{2}}} \\
&\le \frac{2S|kw|^{-\sigma_0} \int_{0}^\infty ( \frac{\pi}{2} +u)^{C_2-\sigma_0 -\frac{1}{2}}e^{-u }du}{(\frac{\pi}{2} - |\arg w|)^{C_2-\sigma_0 +\frac{1}{2}}} \\
&\le \frac{2 e^\frac{\pi}{2} S|kw|^{-\sigma_0}\Gamma(C_2-\sigma_0+\frac{1}{2})}{(\frac{\pi}{2} - |\arg w|)^{C_2-\sigma_0 +\frac{1}{2}}}.   \end{align*}  Last, since $\Re w>0$\[\frac{\pi}{2} -|\arg w| =\arctan\left(\frac{\Re w}{|\Im w|}\right)\geq (\frac{|\Im w| }{\Re w}+1)^{-1}.\]  Putting everything together completes the proof. For part Part (b) apply Assumption (3) \[\ln |\omega_{h,k,n}(z)|\le C_1 k^{s_0}\frac{1}{k}\sum_{r=1}^k |z|^r \Phi(|z|^k,1,\frac{r}{k})\le -C_1 k^{s_0}\ln(1-M_X). \]  

\end{proof}

\section{Proof of Theorem \ref{FourierCoeff}}
The proof is by the circle method.   Fix $\delta>0$ a sufficiently small positive constant and let \[ \alpha=\alpha_n = \frac{ \Re L_{p,q}(z) }{2\pi n^\frac{1}{s_0+1}} \qquad N=N_n= \lfloor \delta n^\frac{1}{s_0+1}\rfloor.\] 

We apply Cauchy's theorem with a contour of radius $e^{-2\pi \alpha_n}$ and parameterize it by $q=e^{-2\pi \alpha_n + i 2\pi \psi}= e^{-\tau}$ for $\psi \in [-1/(N+1), N/(N+1)],$  \[Q_n(z) = \frac{1}{2\pi i}\int_{\Gamma} \frac{P(z,q)}{q^{n+1}}dq= \int_{-\frac{1}{N+1}}^{\frac{N}{N+1}} P(z,e^{-\tau}) e^{n\tau}d\psi.  \]
Break up $[-1/(N+1),N/(N+1)]$ into a series of intervals called Farey arcs \[M_{h,k}=(\frac{h+h'}{k+k'}, \frac{h+h''}{k+k''}]\] where $h'/k'<h/k<h''/k''$ are consecutive elements of Farey fractions of order $N,$ denoted $F_N.$  For the cases of the end points, $h/k \not =0,1,$  we let $M_{1,1}=(-1/(N+1),1/(N+1)]$ and assume $M_{0,1}$ is empty. Using standard arguments, one will demonstrate
\begin{align*}
Q_n(z) & =\sum_{\frac{h}{k} \in F_N} \int_{M_{h,k}}P(z,e^{ -\tau})e^{n\tau} d\tau, \\
  & =\sum_{\frac{h}{k} \in F_N} \int_{- \frac{1}{k(k+k')}}^{\frac{1}{k(k+k'')}}e^{-2\pi in \frac{h}{k}} e^{2\pi n (\alpha_n - iv)}P(z,e^{ -2\pi (\alpha_n - iv) +2\pi i \frac{h}{k}})dv.
\end{align*}

\noindent Recall $P(z,q)$ can be approximated asymptotically as \[P(z,e^{-w+2 \pi i \frac{h}{k}}) =e^{2\pi i n\frac{h}{k} }\omega_{h,k,n}(z)\exp \left(\frac{\Phi_{h,k}(z)}{s_0w^{s_0}}+g_{h,k}(z, w)\right).\] Notice that this representation is useful because we have estimates on $g_{h,k}(z,2\pi(\alpha_n -iv))$ which control its modulus.
\begin{lemma}\label{appliedest}
For every $1\le k \le N_n,$ we can define constants \begin{align*}C_4&=| L_X^M+4\pi |^{-\sigma_0}(\frac{4\pi }{  L^m_X}+1)^{C_2 -\sigma_0 +\frac{1}{2}},\\
  C_5&=\delta^{-s_0+\sigma_0}| L_X^M+\frac{4\pi }{\delta} |^{-\sigma_0}(\frac{4\pi }{\delta  L^m_X}+1)^{C_2 -\sigma_0 +\frac{1}{2}}\end{align*} so that \[|g_{h,k}(z,2\pi(\alpha_n-iv))| \le C_3 n^\frac{\sigma_0}{s_0+1}(C_4 k^{s_0-\sigma_0}+C_5). \]
\end{lemma}
\begin{proof} By the properties of Farey fractions (\cite{MR0422157} for example), control $|v|$ by \[\frac{1}{2kN} \le \frac{1}{k(k+k')} \le \frac{1}{kN}.\] Lemma \ref{Factorization:est} states \begin{align*}|g_{h,k}(z,\alpha_n-iv)| &\le C_3 k^{s_0}|k(2\pi \alpha_n+2\pi |v|) |^{-\sigma_0}(\frac{|v| }{\alpha_n}+1)^{C_2 -\sigma_0 +\frac{1}{2}}\\
& \le C_3k^{s_0}|\frac{kL_X^M}{n^\frac{1}{s_0+1}}+\frac{2\pi }{N} |^{-\sigma_0}(\frac{ 2 \pi n^\frac{1}{s_0+1} }{k N L^m_X}+1)^{C_2 -\sigma_0 +\frac{1}{2}} \\
&\le C_3 n^\frac{\sigma_0}{s_0+1}k^{s_0-\sigma_0}| L_X^M+\frac{4\pi }{k \delta} |^{-\sigma_0}(\frac{4\pi  }{k \delta  L^m_X}+1)^{C_2 -\sigma_0 +\frac{1}{2}}
  \end{align*}
For $1\le k \le \delta^{-1}$ \[|g_{h,k}(z,\alpha_n-iv)| \le C_3 n^\frac{\sigma_0}{s_0+1}\delta^{-s_0+\sigma_0}| L_X^M+\frac{4\pi}{\delta} |^{-\sigma_0}(\frac{4\pi }{\delta  L^m_X}+1)^{C_2 -\sigma_0 +\frac{1}{2}}\]  and $\delta^{-1}\le k\le N_n\le n^\frac{1}{s_0+1}\delta $   \[|g_{h,k}(z,\alpha_n-iv)| \le C_3 n^\frac{\sigma_0}{s_0+1}k^{s_0-\sigma_0}| L_X^M+4\pi  |^{-\sigma_0}(\frac{4\pi }{  L^m_X}+1)^{C_2 -\sigma_0 +\frac{1}{2}}.\] \end{proof}

\noindent If \[ I_{h,k,n}(z) =\int_{-\frac{1}{k(k+k')}}^{\frac{1}{k(k+k'')}}\exp(\bar{\Phi}(z,v;h,k,n)) dv\]
where \begin{align*} \bar{\Phi}(z,v;h,k,n)&=\frac{\Phi_{h,k}(z)}{s_0(2\pi)^{s_0} (\alpha_n-iv)^{s_0}} +2\pi n(\alpha_n-iv) -\frac{s_0+1}{s_0}n^\frac{s_0}{s_0+1}\Re L_{p,q}(z) \\
&+ g_{h,k}(z,2\pi(\alpha_n -iv))  \\
&:= A(z,v;h,k,n) + g_{h,k}(z,2\pi(\alpha_n -iv)) \end{align*}  then by substituting the asymptotic approximation into $P(z,e^{-2\pi (\alpha_n -iv)+2\pi i \frac{h}{k}})$ we can write
\begin{align*}
\exp(-\frac{s_0+1}{s_0}n^\frac{s_0}{s_0+1}\Re L_{p,q}(z))Q_n(z)& = \sum_{\frac{h}{k}\in F_N} \omega_{h,k,n}(z) I_{h,k,n}(z).
\end{align*}
We can control $\Re A(z,v;h,k,n)$ with the following lemma.  
\begin{lemma} \label{criticallemma}
For every $v \in \mathbb{R},$ if $\Re L_{h,k}(z)>0$ then
\[\Re \left(\frac{\Phi_{h,k}(z)}{s_0 (2\pi)^{s_0} (\alpha_n-iv)^{s_0}} +2\pi n(\alpha_n-iv) \right)\le  \frac{(\Re  L_{h,k}(z))^{s_0+1}}{s_0(2\pi \alpha_n)^{s_0}} + 2\pi n \alpha_n \] where equality is uniquely attained at a point $v_0\in \mathbb{R}$ when $\Phi\not <0$ and dually attained at $\pm v_0\in \mathbb{R}$ when $\Phi <0.$
 If $\Re L \le 0 $ then  \[\Re \left(\frac{\Phi_{h,k}(z)}{s_0 (2\pi)^{s_0} (\alpha_n-iv)^{s_0}} +2\pi n(\alpha_n-iv) \right)\le 2\pi n \alpha_n. \]
\end{lemma}
\begin{proof}

Rescale the problem to simplify it \[\Re \left(\frac{\Phi_{h,k}(z) }{s_0(2\pi \alpha_n-i2\pi v)^{s_0}} + (2\pi \alpha_n-i2\pi v)\right) = \frac{|\Phi_{h,k}(z)|}{s_0(2\pi \alpha_n)^{s_0}}  \Re\left(\frac{e^{i(s_0+1)\theta_L} }{(1+i\psi)^{s_0}} \right) + \alpha_n \]  where $\Phi_{h,k}(z) =|\Phi_{h,k}(z)|e^{i(s_0+1)\theta_L},$ $(s_0+1)|\theta_L|\le \pi,$ and $\psi =  -\frac{v}{\alpha}.$ So the lemma simplifies to showing for $|(s_0+1)\theta_L|<\pi/2$ \[\Re\left(\frac{e^{i(s_0+1) \theta_L} }{(1+i\psi)^{s_0}} \right) \le \cos^{s_0+1}(\theta_L)\] and for $\pi/2\le |(s_0+1)\theta_L|<\pi$ \[\Re\left(\frac{e^{i(s_0+1) \theta_L} }{(1+i\psi)^{s_0}} \right) <0.\]
If one writes the expression in polar form \[1+i\psi=|1+i\psi|\exp(i\theta_\psi)\] %and thus,  \[\begin{array}{cc} \tan(\theta_\psi)=\psi & \frac{1}{|1+i\psi|}= \cos( \theta_\psi ). \end{array}\]  
and with some elementary simplifications one will observe\begin{align*}
\Re \left(\frac{e^{i(s_0+1)\theta_L}}{(1+i\psi)^{s_0}}\right) 
%&= \Re\left(\frac{e^{i(m+1)\theta_L}}{(|1+i\psi)|^m e^{im\theta_\psi}} \right) \\.
%& =  \Re \left( \frac{ e^{[i((m+1) \theta_L-m\theta_\psi)]}}{|1+i\psi|^m} \right) \\
%& = \left(\frac{1}{|1+i\psi|}\right)^m \Re  e^{[i((m+1) \theta_L-m\theta_\psi)]} \\
%& = \left(\frac{1}{|1+i\psi|}\right)^m \Re \left[ \cos((m+1) \theta_L-m\theta_\psi) +i \sin((m+1) \theta_L-m\theta_\psi)  \right] \\
&= \cos^{s_0}(\theta_\psi) \cos((s_0+1)\theta_L-s_0\theta_\psi). \end{align*}
Because $\theta_\psi$ is a monotonic function with respect to $\psi$ and maps $ \mathbb{R} \rightarrow (-\pi/2, \pi/2),$ our lemma is equivalent to optimizing the function \[g(\theta)=\cos^{s_0}(\theta) \cos((s_0+1) \theta_L-s_0\theta)\] on the interval $(-\pi/2, \pi/2).$

    When $|(s_0+1)\theta_L|<\pi/2,$ calculus suggests that $g(\theta)$ attains its maximum when $(s_0+1) (\theta -\theta_L) \in \pi\mathbb{Z}.$   When $(s_0+1)\theta_L \not = \pi, -\pi$ then $\theta_L = \theta$ otherwise $\theta=\pm \pi/(s_0+1).$\end{proof}
By Assumptions (4) and (5),  $\{(h,k): L_{h,k}(z) \in [L_{p,q}(z)] \}$ is a finite set. So choose $n$ sufficiently large so that,  $K:=\sup\{k: L_{h,k}(z)\in [L_{p,q}(z)]\}<N_n$ and split the sum accordingly
\begin{align*} \exp(-\frac{s_0+1}{s_0}n^\frac{s_0}{s_0+1}\Re L_{p,q}(z))Q_n(z)&= \sum_{\{(h,k) :\ L_{h,k}(z)\in [L_{p,q}(z)]\}} \omega_{h,k,n}(z)I_{h,k,n}(z) \\
&+ \sum_{\{(h,k) :\ L_{h,k}(z)\not \in [L_{p,q}(z)]\}} \omega_{h,k,n}(z) I_{h,k,n}(z).\end{align*}
We then show that the second term decays exponentially.  In particular we will prove at the end of this section the following Lemma.  
\begin{lemma}\label{minorarccontribution} For every $\delta>0$ sufficiently small, there is an $\eta>0$ so that
\begin{align*}
\sum_{\{(h,k): L_{h,k}(z)\not \in [L_{p,q}(z)]|\}} |\omega_{h,k,n}(z) I_{h,k,n}(z)|\le \exp \left(-n^\frac{s_0}{s_0+1} \eta+C_5 C_3\right) . \\
\end{align*} 
\end{lemma}
\noindent While each $I_{h,k,n}(z)$ can be estimated by saddle point approximation.  To do so, we must must make some reductions.  First, we restrict $\delta>0$ small enough so \[\frac{|\Im L_{h,k}(z)|}{2\pi n^\frac{1}{s_0+1}}\le  \sup\{\frac{|\Im L_{h,k}(z)|}{2\pi n^\frac{1}{s_0+1}}: k\le K, z\in X\} <\frac{1 }{2N K} \le \frac{1 }{2k N}\] and therefore by Lemma \ref{criticallemma} and the properties of the Farey arcs, \begin{align*} I_{h,k,n}(z) &\thicksim \int_{-\frac{1}{2kN}}^{\frac{ 1}{2kN}}\exp(\bar{\Phi}(z,v;h,k,n))dv
\end{align*} with exponentially small relative error.  Next we use Lemma \ref{appliedest}, and note $k\le K$ to demonstrate \[|g_{h,k}(z,\alpha_n-iv)|\le C_6 n^\frac{\sigma_0}{s_0+1}, \qquad C_6 = C_3 (C_4 K^{s_0-\sigma_0}+C_5).\] For $z$ complex, we have $|e^{z}-1|\le e^{|z|} |z|$ and therefore \[\left|\int_{-\frac{1}{2kN}}^{\frac{ 1}{2kN}}\exp(\bar{\Phi}(z,v)) -\exp(A(z,v))dv\right|\le n^\frac{\sigma_0}{s_0+1}C_6e^{C_6}\int_{-\frac{1}{2kN}}^{\frac{ 1}{2kN}}\exp(\Re A(z,v))dv.\] 
Both \[\int_{-\frac{1}{2kN}}^\frac{1}{2kN} \exp(A(z,v;h,k,n))dv,\quad \int_{-\frac{1}{2kN}}^\frac{1}{2kN} \exp(\Re A(z,v;h,k,n))dv \] can be computed using a saddlepoint argument similar to the one given in (e.g. \cite{MR2485091} Page 9-11). When $ \Phi_{h,k}(z)\not \le 0,$  Lemma \ref{criticallemma} implies that \[v_0=-\frac{\Im L_{h,k}(z)}{2\pi n^\frac{1}{s_0+1}}\in (-\frac{1}{2 kN} , \frac{1}{2kN})\] is the saddle point. Applying the saddle point formula produces (1) in Lemma \ref{Ikestimation}. When $\Phi_{h,k}(z) \le 0$  Lemma \ref{criticallemma} demonstrates the existence of two saddle points, $\pm v_0,$ each giving two different contributions.  Applying the saddle point formula to each contribution produces (2) in Lemma \ref{Ikestimation}.

\begin{lemma} \label{Ikestimation}For $\delta>0$ sufficiently small and $(h,k)$ so that $L_{h,k}(z) \in [L_{p,q}(z)]$
 \begin{enumerate}\item If $X\subset \{z: \Phi_{h,k}(z)\not \le 0\}$ is compact and $z\in X$ then, \[I_{h,k,n}(z) \thicksim_X \exp\left(i\frac{s_0+1}{s_0}n^\frac{s_0}{s_0+1} \Im L_{h,k}(z)\right)\sqrt{\frac{ L_{h,k}(z) }{2\pi(s_0+1)n^\frac{s_0+2}{s_0+1} }}\]
 \item If $X\subset \{z: \Phi_{h,k}(z) \le 0\}$ is compact and $z\in X$ then,\begin{align*}I_{h,k,n}(z) &\thicksim_X 2\Re \left[\exp\left(i\frac{s_0+1}{s_0}n^\frac{s_0}{s_0+1} \Im L_{h,k}(z)\right)\sqrt{\frac{ L_{h,k}(z) }{2\pi(s_0+1)n^\frac{s_0+2}{s_0+1} }}\right] \end{align*} \end{enumerate} \end{lemma}  
This completes the proof.  
\subsection{Proof of Lemma \ref{minorarccontribution}}
Consider 
\[
\omega_{h,k,n}(z) I_{h,k,n}(z) = \int_{-\frac{1}{k(k+k')}}^{\frac{1}{k(k+k'')}} \exp(A(z,v;h,k,n)+g_{h,k}(z,\alpha_n-iv)+ \ln \omega_{h,k,n}(z)) dv.\]
We bound $A(z,v;h,k,n)$ using Lemma \ref{criticallemma}.  

If $z\in R(p,q)$ then for every $L_{h,k}(z)\not \in [L_{p,q}(z)],$ $\Re L_{p,q}(z) > \Re L_{h,k}(z).$ Thus, for $X\subset R(p,q)$ compact, apply Lemma \ref{criticallemma}.  For $L_{h,k}(z)\not \in [L_{p,q}(z)]$ and $\Re L_{h,k}(z)>0,$ \begin{align*}A(z,v;h,k,n)&\le  \frac{n^\frac{s_0}{{s_0+1}}}{s_0}\left( \frac{ (\Re L_{h,k}(z))^{{s_0+1}}}{ (\Re L_{p,q}(z))^{s_0}} - \Re L_{p,q}(z) \right)\\
%&\le  \frac{n^\frac{s_0}{{s_0+1}}}{s_0}\left( \Re L_{h,k}(z) - \Re L_{p,q}(z) \right)\\
&\le  -\frac{n^\frac{s_0}{{s_0+1}}}{s_0}\inf_{z\in X} \left( \Re L_{p,q}(z) - \Re L_{h,k}(z) \right)<0 \end{align*} and for $\Re L_{h,k}(z)\le 0,$\begin{align*}A(z,v;h,k,n) &\le -\frac{n^\frac{s_0}{s_0+1}}{s_0}\Re L_{p,q}(z) \le -\frac{n^\frac{s_0}{s_0+1}}{s_0} L^m_X. \end{align*}
By Lemma \ref{appliedest} \[|g_{h,k}(z,\alpha_n+iv)| \le C_3 n^\frac{\sigma_0}{s_0+1}(C_4 k^{s_0-\sigma_0}+C_5)\le C_3 n^\frac{\sigma_0}{s_0+1}(C_4\delta^{s_0-\sigma_0}n^\frac{s_0-\sigma_0}{s_0+1}+C_5) . \]
Along with Assumptions (4),(5), we can conclude that there exists a $\Delta_X>0$ so that if $L_{h,k}(z)\not \in [L_{p,q}(z)],$ then\[|I_{h,k, n}(z)|\le |M_{h,k}|e^{-n^\frac{s_0}{s_0+1}\Delta_X+C_3 n^\frac{\sigma_0}{s_0+1}(C_4\delta^{s_0-\sigma_0}n^\frac{s_0-\sigma_0}{s_0+1}+C_5)}.\]
By Lemma \ref{Factorization:est} \begin{equation} \ln |\omega_{h,k,n}(z)| \le -C_1 k^{s_0}\ln (1-M_X) \le -C_1\delta^{s_0}n^\frac{s_0}{s_0+1}\ln(1-M_X).  \end{equation} 
These estimates now prove
\[\left| \omega_{h,k,n}(z) I_{h,k,n}(z)\right|\le |M_{h,k}|\exp(-n^\frac{s_0}{s_0+1} (\Delta_X-C_3C_4\delta^{s_0-\sigma_0}+C_1\delta^{s_0}\ln (1-M_X)) +C_5 C_3 ).\] 

Observe $C_5$ is the only constant that is dependent on choice of $\delta,$ and so if we set \[\eta=\Delta_X-C_3C_4\delta^{s_0-\sigma_0}+C_1\delta^{s_0}\ln (1-M_X)\]  and by making $\delta$ small enough, we can require $\eta>0. $  
\begin{align*}
|\omega_{h,k,n}(z) I_{h,k,n}(z)| &\le |M_{h,k}|\exp \left(- n^\frac{s_0}{s_0+1}\eta +C_5 C_3 \right).
\end{align*} 

Now we apply this upper bound to every ``minor arc" uniformly and we note $\sum_{h/k\in F_N} |M_{h,k}|=1.$
\begin{align*}
\sum_{\{(h,k): L_{h,k}(x)\not \in [L_{p,q}(x)]\}} |\omega_{h,k,n}(z) I_{h,k,n}(z) | &\le \exp \left(- n^\frac{s_0}{s_0+1}\eta +C_5C_3\right).
\end{align*}

\begin{remark}  The main theorem in this paper is a simplified and streamlined version of a result given in a chapter in the author's doctoral thesis \cite{ParryDissertation}.

\end{remark}
 \bibliography{mybib}{}
\bibliographystyle{amsplain}

  \end{document}